\documentclass[11pt]{article}
\usepackage{amsfonts}
\usepackage{amssymb,amsfonts,amsmath,amsthm,cite}
\usepackage{enumerate}
\usepackage{ytableau}

\usepackage{graphicx}
\usepackage{tikz}
\usetikzlibrary{calc,decorations.pathreplacing}
\usetikzlibrary{arrows,shapes,positioning}
\usetikzlibrary{decorations.markings}

\tikzset{middlearrow/.style={
        decoration={markings,
            mark= at position 0.5 with {\arrow{#1}} ,
        },
        postaction={decorate}
    }
}

\allowdisplaybreaks[4]

\theoremstyle{plain}
\newtheorem{theorem}{Theorem}[section]

\theoremstyle{definition}
\newtheorem{definition}[theorem]{Definition}

\theoremstyle{remark}
\newtheorem{remark}[theorem]{Remark}

\makeatletter \@addtoreset{equation}{section}

\begin{document}

\begin{center}
{\Large {\bf SIP classes and four-parameter partition identities}}

\vskip 6mm

{\small Runqiao Li\\[2mm]

Department of Mathematics\\
University of Texas Rio Grande Valley\\
Edinburg, TX 78541, USA\\[3mm]

runqiao.li@utrgv.edu or runqaoli@outlook.com\\[2mm]}
\end{center}

\noindent {\bf Abstract.}
The four-parameter weight of partitions played an important role in the theory of integer partitions, for its connection with various statistics, including the alternating sum and the BG-rank. In 2022, Andrews introduced the SIP classes, by which he reviewed a number of classic partition identities and provided new combinatorial insights. In this work, we extend the SIP classes and provide a unified method to study the four-parameter weight of partitions. By treating partitions with position parity as examples, we provide four-parameter partition identities related to these partition sets. And as corollary, we also present the generating functions that keep track of the BG-rank and the joint distribution of the number of odd parts and the alternating sum, respectively.



\section{Introduction}

A partition is a finite weakly decreasing sequence of positive integers $\lambda=(\lambda_1,\lambda_2,\ldots,\lambda_{\ell})$, and it is called a strict partition if the sequence is strictly decreasing. Each entry $\lambda_i$ for $1\leq i\leq\ell$ is called a part of $\lambda$, where $\lambda_1,\lambda_3,\lambda_5,\ldots$ are the odd-indexed parts, $\lambda_2,\lambda_4,\lambda_6,\ldots$ are the even-indexed parts. The number of parts is called the length of $\lambda$, denoted by $\ell(\lambda)$. The weight of a partition, denoted by $|\lambda|$, is defined as the sum of its parts. And if $|\lambda|=n$, we say that $\lambda$ is a partition of $n$ and denote by $\lambda\vdash n$. The set of all partitions will be denoted by $\mathcal{P}$.

Through out this paper, we shall also use the $q$-Porchhamer symbol defined as follows.
$$(a;q)_{0}:=1,\quad(a;q)_n:=\prod_{i=0}^{n-1}(1-aq^i)\quad\text{and}\quad
(a;q)_{\infty}:=\lim_{n\to\infty}(a;q)_{n},
$$
where $a$ and $q$ are complex variables and the infinite product is convergent for $|q|<1$. We also adopt the abbreviation
$$(a_1,a_2,\ldots,a_k;q)_{n}:=(a_1;q)_n(a_2;q)_n\cdots(a_k;q)_n$$
and
$$(a_1,a_2,\ldots,a_k;q)_{\infty}:=(a_1;q)_{\infty}(a_2;q)_{\infty}\cdots(a_k;q)_{\infty},$$
where $a_i$'s are complex variables.

In the theory of partitions, generating functions are the bridge that connect the combinatorial and the analytic aspects. By considering various partition statistics, we usually get multivariable extensions of partition identities. In 2011, Savage and Sills~\cite{SavageSills} provided new sum sides for the little-Göllnitz identities and their bivariable generalizations. Let $\mathcal{G}_1$ be the set of strict partitions with even-indexed parts being even, and $\mathcal{G}_2$ be the set of strict partitions with odd-indexed parts being even.  Let $o(\lambda)$ be the number of odd parts in a partition $\lambda$, they have proved the following.
\begin{theorem}[Savage and Sills, 2011]\label{thm:SavageSillsBivariable}
\begin{equation}
\sum_{\lambda\in\mathcal{G}_1}x^{o(\lambda)}q^{|\lambda|}=\sum_{n=0}^{\infty}\frac{x^{n}q^{\binom{2n}{2}}(-q/x;q^4)_n}{(q^2;q^2)_{2n}}=\frac{(-xq;q^4)_{\infty}}{(q^2;q^4)_{\infty}},    
\end{equation}
\begin{equation}
\sum_{\lambda\in\mathcal{G}_1}x^{o(\lambda)}q^{|\lambda|}=\sum_{n=0}^{\infty}\frac{x^{n}q^{\binom{2n+1}{2}}(-1/xq;q^4)_n}{(q^2;q^2)_{2n}}=\frac{(-xq^3;q^4)_{\infty}}{(q^2;q^4)_{\infty}}.     
\end{equation}
\end{theorem}
Recently, the author in~\cite{Li} revisited these partition sets and considered a different statistic. For a partition $\lambda=(\lambda_1,\lambda_2,\ldots,\lambda_{\ell})$, the alternating sum is defined as
$$a(\lambda):=\lambda_1-\lambda_2+\lambda_3-\cdots+(-1)^{\ell-1}\lambda_{\ell}.$$
\begin{theorem}[Li, 2025]\label{thm:LiBivariable}
\begin{equation}\label{G1AlternatingSum}
\sum_{\lambda\in\mathcal{G}_1}z^{a(\lambda)}q^{|\lambda|}=\sum_{n=0}^{\infty}\frac{z^nq^{\binom{2n}{2}}(-zq;q^4)_{n}}{(q^4;q^4)_{n}(z^2q^2;q^4)_{n}}=\frac{(-zq;q^4)_{\infty}}{(z^2q^2;q^4)_{\infty}}=\frac{1}{(zq;q^4)_{\infty}(z^2q^6;q^8)_{\infty}},    
\end{equation}
\begin{equation}\label{G2AlternatingSum}
\sum_{\lambda\in\mathcal{G}_2}z^{a(\lambda)}q^{|\lambda|}=\sum_{n=0}^{\infty}\frac{z^{n}q^{\binom{2n+1}{2}}(-zq^{-1};q^4)_{n}}{(q^4;q^4)_{n}(z^2q^2;q^4)_{n}}=\frac{(-zq^3;q^4)_{\infty}}{(z^2q^2;q^4)_{\infty}}=\frac{1}{(zq^3;q^4)_{\infty}(z^2q^2;q^8)_{\infty}}.    
\end{equation}    
\end{theorem}
It is natural to ask if we can unify the identities above. That is, for each of $\mathcal{G}_1$ and $\mathcal{G}_2$, we wish to find a three-variable generating function that keeps track of the number of odd parts, the alternating sum, and the weight of the partition, simultaneously. 

Such a question was first considered by Andrews for the set of all partitions~\cite{AndrewsThreeVariable}.
\begin{theorem}[Andrews, 2004]
\begin{equation}\label{eq:AndrewsThreeVariable}
\sum_{\lambda\in\mathcal{P}}x^{o(\lambda)}z^{a(\lambda)}q^{|\lambda|}=\frac{(-xzq;q^2)_{\infty}}{(x^2q^2;q^4)_{\infty}(z^2q^2;q^4)_{\infty}(q^4;q^4)_{\infty}}
\end{equation}    
\end{theorem}
\begin{remark}
In his paper, instead of alternating sum, Andrews considered the number of odd parts in the conjugate of $\lambda$. But by a combinatorial argument one can see that these two statistics are equivalent.  
\end{remark}
Later, Boulet~\cite{Boulet} introduced a four-parameter weight on integer partitions and presented partition identities related to it. Given a partition $\lambda=(\lambda_1,\lambda_2,\lambda_3,\ldots,\lambda_{\ell})$, the four-parameter weight $\omega(\lambda)$ is defined as
$$\omega(\lambda):=a^{\sum\lceil\frac{\lambda_{2i-1}}{2}\rceil}b^{\sum\lfloor\frac{\lambda_{2i-1}}{2}\rfloor}c^{\sum\lceil\frac{\lambda_{2i}}{2}\rceil}d^{\sum\lfloor\frac{\lambda_{2i}}{2}\rfloor},$$
where each sum runs through all the positive integers $i$, and we take $\lambda_i=0$ for $i>\ell$. Graphically, this weight means we fill the Ferrers graph of $\lambda$ as follows.\begin{center}
\ytableausetup
{mathmode,boxframe=normal,boxsize=1.3em}
\begin{ytableau}
a & b & a & b & a & b & a\\
c & d & c & d & c & d\\
a & b & a & b\\
c & d
\end{ytableau}
\end{center}
\begin{theorem}[Boulet, 2006]
The four-parameter generating function for the set of all partitions is
\begin{equation}\label{eq:BouletP}
\sum_{\lambda\in\mathcal{P}}\omega(\lambda)=\frac{(-a;Q)_{\infty}(-abc;Q)_{\infty}}{(ab;Q)_{\infty}(ac;Q)_{\infty}(Q;Q)_{\infty}},
\end{equation}
where $Q:=abcd$.    
\end{theorem}
Note that by the definition of $\omega(\lambda)$, to get $x^{o(\lambda)}z^{a(\lambda)}q^{|\lambda|}$, we simply need to set $a\to xzq$, $b\to zq/x$, $c\to xq/z$ and $d\to q/xz$. So,~\eqref{eq:AndrewsThreeVariable} becomes a corollary of~\eqref{eq:BouletP}, and if we can get the four-parameter weight generating functions for $\mathcal{G}_1$ and $\mathcal{G}_1$, we would have the unification of Theorem~\ref{thm:SavageSillsBivariable} and Theorem~\ref{thm:LiBivariable} immediately.

It turns out the separable integer partition (SIP) classes can be generalized to fit the study of the four-parameter weight. Besides $\mathcal{G}_1$ and $\mathcal{G}_1$, we shall also treat the following partitions sets, so to better demonstrate this method. Let $\mathcal{P}_1$ be the set of ordinary partitions with even-indexed parts being even, and $\mathcal{P}_2$ be the set of ordinary partitions with odd-indexed parts being even. Our main result is the following.
\begin{theorem}
The following four-parameter partition identities hold.
\begin{equation}
\sum_{\lambda\in\mathcal{G}_{1}}\omega(\lambda)=\sum_{n=0}^{\infty}\frac{a^{n}Q^{\binom{n}{2}}(-b;Q)_{n}}{(ab;Q)_{n}(Q;Q)_{n}}=\frac{(-a;Q)_{\infty}}{(ab;Q)_{\infty}},
\end{equation}
\begin{equation}
\sum_{\lambda\in\mathcal{G}_{2}}\omega(\lambda)=\sum_{n=0}^{\infty}\frac{d^{-n}Q^{\binom{n+1}{2}}(-c^{-1};Q)_{n}}{(ab;Q)_{n}(Q;Q)_n}=\frac{(-abc;Q)_{\infty}}{(ab;Q)_{\infty}},
\end{equation}
\begin{equation}
\sum_{\lambda\in\mathcal{P}_{1}}\omega(\lambda)=\sum_{n=0}^{\infty}\frac{Q^{n}(-aQ^{-1};Q)_{n}}{(ab;Q)_{n}(Q;Q)_{n}}+\sum_{n=0}^{\infty}\frac{abQ^{n}(-a;Q)_{n}}{(ab;Q)_{n+1}(Q;Q)_{n}}=\frac{(-a;Q)_{\infty}}{(ab;Q)_{\infty}(Q;Q)_{\infty}},
\end{equation}
\begin{equation}
\sum_{\lambda\in\mathcal{P}_{2}}\omega(\lambda)=\sum_{n=0}^{\infty}\frac{Q^{n}(-1/d;Q)_{n}}{(ab;Q)_{n}(Q;Q)_{n}}+\sum_{n=0}^{\infty}\frac{abQ^{n}(-Q/d;Q)_{n}}{(ab;Q)_{n+1}(Q;Q)_{n}}=\frac{(-Q/d;Q)_{\infty}}{(ab;Q)_{\infty}(Q;Q)_{\infty}}.
\end{equation}
\end{theorem}

The rest of this paper is organized as follows. In Section~\ref{sec:SIP}, we introduce the SIP classes and a four-parameter generalization of it. In Section~\ref{sec:Strict} and Section~\ref{sec:Ordinary}, we prove the four-parameter identities for $\mathcal{G}_1$, $\mathcal{G}_2$ and  $\mathcal{P}_1$, $\mathcal{P}_2$, respectively. Finally, in Section~\ref{sec:Comment}, we present the corollaries derived from our four-parameter identities. We also give a brief discussion for some potential applications of the method we introduced, which might lead to some further study.

\section{Separable integer partition classes}\label{sec:SIP}

The SIP class waw first introduced by Andrews in~\cite{AndrewsSIP}. 

\begin{definition}
Given a set of partitions $\mathcal{S}$, if there is a subset $\mathcal{B}\subset\mathcal{S}$ and an integer $k>0$ such that for each partition $\lambda\in\mathcal{S}$ there is an unique pair of partitions $(\beta,\mu)$, where $\beta\in\mathcal{B}$ with $\ell(\beta)=\ell(\lambda)$, $\mu$ is a partition into multiples of $k$ with $\ell(\mu)\leq\ell(\lambda)$, and $|\lambda_i|=|\beta_i|+|\mu_i|$ for all $i$, then we say $\mathcal{S}$ is an SIP class with basis $\mathcal{B}$ and modulus $k$. Here we take $\mu_i=0$ if $i>\ell(\mu)$.   
\end{definition}

\begin{remark}
The definition adopted by Andrews is tailored for partitions with certain gap conditions. Although it looks different from what we present here, they are equivalent essentially. 
\end{remark}
The SIP class gives a natural way to attain the generating function of the set $\mathcal{S}$. By counting partitions based on their length,
$$\sum_{\lambda\in\mathcal{S}}q^{|\lambda|}=\sum_{n=0}^{\infty}\sum_{\substack{\lambda\in\mathcal{S}\\\ell(\lambda)=n}}q^{|\lambda|}.$$
Now for each $n$, the inner sum can be written as
$$\sum_{\substack{\lambda\in\mathcal{S}\\\ell(\lambda)=n}}q^{|\lambda|}=B(n;q)\times\frac{1}{(q^k;q^k)_n},$$
where $B(n;q)$ is the generating function for partitions in the basis $\mathcal{B}$ with length $n$. Since partitions with parts divisible by $k$ and length bounded by $n$ are clearly generated by $1/(q^k;q^k)_{n}$, this immediately follows from the correspondence between $\lambda$ and $(\beta,\mu)$. So, we conclude the following.
\begin{theorem}[Andrews, 2022]\label{thm:SIPqGF}
If $\mathcal{S}$ is an SIP class with basis $\mathcal{B}$ and modulus $k$,
\begin{equation}\label{eq:SIPqGF}
\sum_{\lambda\in\mathcal{S}}q^{|\lambda|}=\sum_{n=0}^{\infty}\frac{B(n;q)}{(q^k;q^k)_{n}}.
\end{equation}
\end{theorem}
We wish to extend this expression to study the four-parameter weight $\omega(\lambda)$. For that, it is better to understand the construction of SIP classes from a combinatorial perspective. Let's take $\mathcal{G}_1$ as an example, and we claim that it is an SIP class of modulus $2$. Given a partition $\lambda\in\mathcal{G}_1$, we keep subtracting columns of width $2$ from its Ferrers graph till the gap between any two successive rows are either $1$ or $2$. We take the remained portion to be the Ferrers graph of $\beta$, and take all the subtracted columns together as the Ferrers graph of $\mu$. This can be shown as follows.
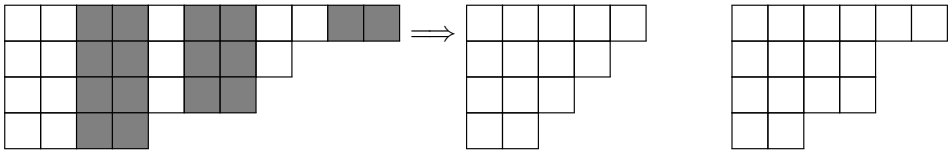
\begin{center}
\ytableausetup{mathmode,boxframe=normal,boxsize=1.2em}
\begin{ytableau}
\empty & \empty & *(gray)\empty & *(gray)\empty & \empty & *(gray)\empty & *(gray)\empty & \empty & \empty & *(gray)\empty & *(gray)\empty \\
\empty & \empty & *(gray)\empty & *(gray)\empty & \empty & *(gray)\empty & *(gray)\empty & \empty \\
\empty & \empty & *(gray)\empty & *(gray)\empty & \empty & *(gray)\empty & *(gray)\empty\\
\empty & \empty & *(gray)\empty & *(gray)\empty
\end{ytableau}
$\Longrightarrow$
\ydiagram{5,4,3,2}
\hspace{1cm}
\ydiagram{6,4,4,2}
\end{center}
It is clear that $\mu$ is a partition with only even parts, and $\lambda_i=\beta_i+\mu_i$. Let $\mathcal{B}_{\mathcal{G}_1}$ be the set of partitions with even-indexed parts being even and $1\leq\lambda_i-\lambda_{i+1}\leq2$, then this set will be the basis and $\beta\in\mathcal{B}_{\mathcal{G}_1}$. So, we have the desired basis and modulus for $\mathcal{G}_1$ to be an SIP classes. With a similar argument, one can see that $\mathcal{G}_2$, $\mathcal{P}_1$ and $\mathcal{P}_2$ are also SIP classes of modulus $2$, and we can also find the basis for each of them.

Now, for the four-parameter weight, if $\mathcal{S}$ is an SIP classes with basis $\mathcal{B}$ and modulus $2$, then the columns we subtracted are of the following types.
\begin{center}
    \begin{ytableau}
    a & b\\
    c & d\\
    \none[\vdots]& \none[\vdots]\\
    a & b\\
    c & d
    \end{ytableau}
    \ or\ 
    \begin{ytableau}
    b & a\\
    d & c\\
    \none[\vdots]& \none[\vdots]\\
    b & a\\
    d & c\\
    \end{ytableau}
    \quad\quad\quad
    \begin{ytableau}
    a & b\\
    c & d\\
    \none[\vdots]& \none[\vdots]\\
    a & b\\
    \end{ytableau}
    \ or\ 
    \begin{ytableau}
    b & a\\
    d & c\\
    \none[\vdots]& \none[\vdots]\\
    b & a\\
    \end{ytableau}
\end{center}
Let $B(n;a,b,c,d)$ be the four-parameter generating function for the partitions in the basis with length $n$. Theorem~\ref{thm:SIPqGF} has the following extension when $k=2$.
\begin{theorem}\label{Thm:SIP4}
If $\mathcal{S}$ is an SIP classes with basis $\mathcal{B}$ and modulus $2$,
\begin{equation}
\sum_{\lambda\in\mathcal{S}}\omega(\lambda)=\sum_{n=0}^{\infty}\frac{B(2n;a,b,c,d)}{(ab,Q)_{n}(Q;Q)_{n}}+\sum_{n=0}^{\infty}\frac{B(2n+1;a,b,c,d)}{(ab,Q)_{n+1}(Q;Q)_{n}}.
\end{equation}
\end{theorem}

Thus, to get the four-parameter generating function for $\mathcal{S}$, it suffices to consider partitions in the basis $\mathcal{B}$. And this will be the main method for the rest of this paper. For the sake of completeness, we also present the identities that will be needed in the sequel.

The first one is a finite form of the $q$-binomial theorem. By~\cite[Equation (3.3.6)]{A},
\begin{equation}\label{eq:qBinomialFinite}
(-z;q)_{n}=\sum_{k=0}^{n}z^{k}q^{\binom{k}{2}}{n\brack k}_{q}.
\end{equation}
The next one is Heine's $q$-analogue of Gauss' summation formula~\cite[Equation (1.5.1)]{BHS}.
\begin{theorem}[$q$-Gauss Summation]
For $|q|<1$ and $|c/ab|<1$,
\begin{equation}\label{eq:qGauss}
\sum_{n=0}^{\infty}\frac{(a;q)_{n}(b;q)_{n}}{(q;q)_{n}(c;q)_{n}}\left(\frac{c}{ab}\right)^{n}=\frac{(c/a;q)_{\infty}(c/b;q)_{\infty}}{(c;q)_{\infty}(c/ab;q)_{\infty}}
\end{equation}
\end{theorem}
Finally, we shall also need the following recurrences for the $q$-binomial coefficients.
\begin{theorem}[Recurrence for the $q$-binomial coefficients] For $n\geq m\geq0$,
\begin{equation}\label{QbinRec1}
{n\brack m}_{q}=q^{m}{n-1\brack m}_{q}+{n-1\brack m-1}_{q},
\end{equation}
\begin{equation}\label{QbinRec2}
{n\brack m}_{q}={n-1\brack m}_{q}+q^{n-m}{n-1\brack m-1}_{q}.   
\end{equation}
\end{theorem}

\section{Strict partitions with position parity}\label{sec:Strict}

In this section, we present the four-parameter generating function for strict partitions with position parity. Recall that $\mathcal{G}_1$ and $\mathcal{G}_2$ are the sets of strict partitions with even-indexed and odd-indexed parts being even, respectively. Both $\mathcal{G}_1$ and $\mathcal{G}_2$ are separable partition classes with modulus $2$.

\subsection{Strict partitions with even indexed parts being even}

We have seen that $B_{\mathcal{G}_1}$, the basis for $\mathcal{G}_1$, consists of partitions in $\mathcal{G}_1$ such that $1\leq\lambda_i-\lambda_{i+1}\leq2$. Let $B_{\mathcal{G}_1}(n,h):=B_{\mathcal{G}_1}(a,b,c,d;n,h)$ be the four-parameter generating function of partitions in $B_{\mathcal{G}_1}$ with length $n$ and largest part $h$. 
\begin{theorem}
The following initial values are true for $B_{\mathcal{G}_1}(n,h)$.
$$B_{\mathcal{G}_1}(0,h)=\left\{\begin{array}{cc}
   1  & \text{if $h=0$}, \\
   0  & \text{otherwise},
\end{array}\right.$$
and
$$B_{\mathcal{G}_1}(1,h)=\left\{\begin{array}{cc}
   a  & \text{if $h=1$},\\
   ab & \text{if $h=2$},\\
   0  & \text{otherwise}.
\end{array}\right.$$
\end{theorem}

\begin{theorem}\label{thm:RecG1}
The following recurrence relations hold for $B_{\mathcal{G}_1}(a,b,c,d;n,h)$.
\begin{equation}\label{Recg11}
B_{\mathcal{G}_1}(n,2h)=bB_{\mathcal{G}_1}(n,2h-1)
\end{equation}
and
\begin{equation}\label{Recg12}
B_{\mathcal{G}_1}(n,2h-1)=aQ^{h-1}B_{\mathcal{G}_1}(n-2,2h-3)+abQ^{h-1}B_{\mathcal{G}_1}(n-2,2h-5).
\end{equation}
\end{theorem}
\begin{proof}
We first show \eqref{Recg11}. In $\mathcal{B}_{\mathcal{G}_1}$ all the even-indexed parts must be even. So, if the largest part is even, then the gap between $\lambda_1$ and $\lambda_2$ must be $2$, and the last box in the first row must be filled with a letter $b$. Thus, if we delete it from $\lambda_1$, we end up with a partition in the basis with the same length and the largest part being $1$ less than $\lambda$. This is shown as follows.
\begin{center}
   \begin{ytableau}
    a & b & \none[\cdots] & a & b & a & *(gray)b\\
    c & d & \none[\cdots] & c & d\\
    \none[\vdots] & \none[\vdots] & \none & \none[\vdots]
    \end{ytableau}
    $\Longrightarrow$
    \begin{ytableau}
    a & b & \none[\cdots] & a & b & a \\
    c & d & \none[\cdots] & c & d \\
    \none[\vdots] & \none[\vdots] & \none & \none[\vdots]
    \end{ytableau}
\end{center}

Next, for \eqref{Recg12}, we consider two different cases. If the largest part is $2h-1$ then the second largest part must be $2h-2$. By deleting them from the Ferrers diagram, we get a partition in the basis with length $n-2$ and largest part being $2h-3$ or $2h-4$. For the latter case, we repeat the argument for \eqref{Recg11}, and that makes the largest part now being $2h-5$. The following diagrams explained this process.
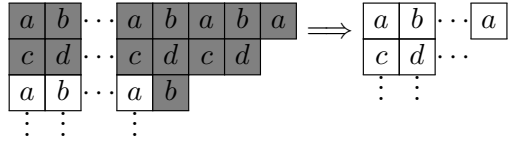
\begin{center}
\begin{ytableau}
    *(gray)a & *(gray)b & \none[\cdots] & *(gray)a & *(gray)b & *(gray)a & *(gray)b & *(gray)a \\
    *(gray)c & *(gray)d & \none[\cdots] & *(gray)c & *(gray)d & *(gray)c & *(gray)d\\
    a & b & \none[\cdots] & a & b & a \\
    \none[\vdots] & \none[\vdots] & \none & \none[\vdots]
    \end{ytableau}
    $\Longrightarrow$
    \begin{ytableau}
    a & b & \none[\cdots] & a & b & a  \\
    c & d & \none[\cdots] & c & d \\
    \none[\vdots] & \none[\vdots] & \none & \none[\vdots]
    \end{ytableau}\\
    \begin{ytableau}
    *(gray)a & *(gray)b & \none[\cdots] & *(gray)a & *(gray)b & *(gray)a & *(gray)b & *(gray)a \\
    *(gray)c & *(gray)d & \none[\cdots] & *(gray)c & *(gray)d & *(gray)c & *(gray)d\\
    a & b & \none[\cdots] & a & *(gray)b \\
    \none[\vdots] & \none[\vdots] & \none & \none[\vdots]
    \end{ytableau}
    $\Longrightarrow$
    \begin{ytableau}
    a & b & \none[\cdots] & a   \\
    c & d & \none[\cdots] \\
    \none[\vdots] & \none[\vdots]
    \end{ytableau}
\end{center}
    So, we finish the proof.
\end{proof}
By solving the recurrence relation, we have the following expressions for $B_{\mathcal{G}_1}(n,h)$.
\begin{theorem}
For any positive integer $n$ and $h$,
\begin{equation}\label{bg1oo}
B_{\mathcal{G}_1}(2n-1,2h-1)=Q^{\binom{n}{2}+\binom{h-n+1}{2}}a^{n}b^{h-n}{n-1\brack h-n}_{Q},
\end{equation}
\begin{equation}\label{bg1eo}
B_{\mathcal{G}_1}(2n,2h-1)=Q^{\binom{n+1}{2}+\binom{h-n}{2}}a^{n}b^{h-n-1}{n-1\brack h-n-1}_{Q},
\end{equation}
\begin{equation}\label{bg1oe}
B_{\mathcal{G}_1}(2n-1,2h)=Q^{\binom{n}{2}+\binom{h-n+1}{2}}a^{n}b^{h-n+1}{n-1\brack h-n}_{Q},
\end{equation}
\begin{equation}\label{bg1ee}
B_{\mathcal{G}_1}(2n,2h)=Q^{\binom{n+1}{2}+\binom{h-n}{2}}a^{n}b^{h-n}{n-1\brack h-n-1}_{Q}.
\end{equation}

\end{theorem}

\begin{proof}
We first show that the right hand sides satisfy the recurrence relations. Define
$$R_{\mathcal{G}_1}(2n-1,2h-1):=Q^{\binom{n}{2}+\binom{h-n+1}{2}}a^{n}b^{h-n}{n-1\brack h-n}_{Q},$$
then
\begin{align*}
R_{\mathcal{G}_1}(2n-1,2h-1)=&Q^{\binom{n}{2}+\binom{h-n+1}{2}}a^{n}b^{h-n}{n-1\brack h-n}_{Q}\\
=&Q^{\binom{n}{2}+\binom{h-n+1}{2}}a^{n}b^{h-n}\left(Q^{h-n}{n-2\brack h-n}_{Q}+{n-2\brack h-n-1}_{Q}\right)\\
=&Q^{\binom{n-1}{2}+\binom{h-n+1}{2}+h-1}a^{n}b^{h-n}{n-2\brack h-n}_{Q}\\
&+Q^{\binom{n-1}{2}+\binom{h-n}{2}+h-1}a^{n}b^{h-n}{n-2\brack h-n-1}_{Q}\\
=&Q^{h-1}aR_{\mathcal{G}_1}(2n-3,2h-3)+Q^{h-1}abR_{\mathcal{G}_1}(2n-3,2h-5).
\end{align*}
Similarly, define
$$R_{\mathcal{G}_1}(2n,2h-1):=Q^{\binom{n+1}{2}+\binom{h-n}{2}}a^{n}b^{h-n-1}{n-1\brack h-n-1}_{Q},$$
then
\begin{align*}
R_{\mathcal{G}_1}(2n,2h-1)=&Q^{\binom{n+1}{2}+\binom{h-n}{2}}a^{n}b^{h-n-1}{n-1\brack h-n-1}_{Q}\\
=&Q^{\binom{n+1}{2}+\binom{h-n}{2}}a^{n}b^{h-n-1}\left(Q^{h-n-1}{n-2\brack h-n-1}_{Q}+{n-2\brack h-n-2}_{Q}\right)\\
=&Q^{\binom{n}{2}+\binom{h-n}{2}+h-1}a^{n}b^{h-n-1}{n-2\brack h-n-1}_{Q}\\
&+Q^{\binom{n}{2}+\binom{h-n-1}{2}+h-1}a^{n}b^{h-n-1}{n-2\brack h-n-2}_{Q}\\
=&Q^{h-1}aR_{\mathcal{G}_1}(2n-2,2h-3)+Q^{h-1}abR_{\mathcal{G}_1}(2n-4,2h-5).
\end{align*}
We have proved (\ref{bg1oo}) and (\ref{bg1eo}). Note that (\ref{bg1oe}) and (\ref{bg1ee}) follow from (\ref{Recg11}). By checking the initial values, we finish the proof.
\end{proof}

Now we are ready to present the four-parameter identity related to partitions in $\mathcal{G}_1$.
\begin{theorem}\label{Thm:G1Four}
The four-parameter generating function for $\mathcal{G}_1$ is given by
$$\sum_{\lambda\in\mathcal{G}_{1}}\omega(\lambda)=\sum_{n=0}^{\infty}\frac{a^{n}Q^{\binom{n}{2}}(-b;Q)_{n}}{(ab;Q)_{n}(Q;Q)_{n}}=\frac{(-a;Q)_{\infty}}{(ab;Q)_{\infty}}.$$
\end{theorem}

\begin{proof}
By \ref{Thm:SIP4},
\begin{align*}
\sum_{\lambda\in\mathcal{G}_{1}}\omega(\lambda)=&1+
\sum_{n=1}^{\infty}\sum_{h=0}^{\infty}\frac{B_{\mathcal{G}_1}(2n,2h-1)+B_{\mathcal{G}_1}(2n,2h)}{(ab;Q)_n(Q;Q)_n}\\
&+\sum_{n=1}^{\infty}\sum_{h=0}^{\infty}\frac{B_{\mathcal{G}_1}(2n-1,2h-1)+B_{\mathcal{G}_1}(2n-1,2h)}{(ab;Q)_n(Q;Q)_{n-1}}.
\end{align*}   
For the first double sum, which counts all such partitions with even length,
\begin{align*}
&\sum_{n=1}^{\infty}\sum_{h=0}^{\infty}\frac{B_{\mathcal{G}_1}(2n,2h-1)+B_{\mathcal{G}_1}(2n,2h)}{(ab;Q)_n(Q;Q)_n}\\
=&\sum_{n=1}^{\infty}\frac{1}{(ab;Q)_n(Q;Q)_n}\sum_{h=0}^{\infty}\left(Q^{\binom{n+1}{2}+\binom{h-n}{2}}a^{n}b^{h-n-1}{n-1\brack h-n-1}_{Q}\right.\\
&\left.+Q^{\binom{n+1}{2}+\binom{h-n}{2}}a^{n}b^{h-n}{n-1\brack h-n-1}_{Q}\right)\\
=&\sum_{n=1}^{\infty}\frac{Q^{\binom{n+1}{2}}a^{n}(1+b)}{(ab;Q)_n(Q;Q)_n}\sum_{h=0}^{\infty}Q^{\binom{h+1}{2}}b^{h}{n-1\brack h}_{Q}\\
&\text{(By~\eqref{eq:qBinomialFinite} with $z\to bQ$ and $q\to Q$)}\\
=&\sum_{n=1}^{\infty}\frac{Q^{\binom{n+1}{2}}a^{n}(1+b)(-bQ;Q)_{n-1}}{(ab;Q)_n(Q;Q)_n}\\
=&\sum_{n=1}^{\infty}\frac{Q^{\binom{n+1}{2}}a^{n}(-b;Q)_{n}}{(ab;Q)_n(Q;Q)_n}.
\end{align*}
And for the second double sum, which corresponds to all such partitions with odd length,
\begin{align*}
&\sum_{n=1}^{\infty}\sum_{h=0}^{\infty}\frac{B_{\mathcal{G}_1}(2n-1,2h-1)+B_{\mathcal{G}_1}(2n-1,2h)}{(ab;Q)_n(Q;Q)_{n-1}}\\
=&\sum_{n=1}^{\infty}\frac{1}{(ab;Q)_n(Q;Q)_{n-1}}\sum_{h=0}^{\infty}\left(Q^{\binom{n}{2}+\binom{h-n+1}{2}}a^{n}b^{h-n}{n-1\brack h-n}_{Q}\right.\\
&\left.+Q^{\binom{n}{2}+\binom{h-n+1}{2}}a^{n}b^{h-n+1}{n-1\brack h-n}_{Q}\right)\\
=&\sum_{n=1}^{\infty}\frac{Q^{\binom{n}{2}}a^{n}(1+b)}{(ab;Q)_n(Q;Q)_{n-1}}\sum_{h=0}^{\infty}Q^{\binom{h+1}{2}}b^{h}{n-1\brack h}_{Q}\\
&\text{(By~\eqref{eq:qBinomialFinite} with $z\to bQ$ and $q\to Q$)}\\
=&\sum_{n=1}^{\infty}\frac{Q^{\binom{n}{2}}a^{n}(1+b)(-bQ;Q)_{n-1}}{(ab;Q)_n(Q;Q)_{n-1}}\\
=&\sum_{n=1}^{\infty}\frac{Q^{\binom{n}{2}}a^{n}(-b;Q)_{n}}{(ab;Q)_n(Q;Q)_{n-1}}.
\end{align*}
Hence,
\begin{align*}
\sum_{\lambda\in\mathcal{G}_{1}}\omega(\lambda)=&1+\sum_{n=1}^{\infty}\frac{Q^{\binom{n+1}{2}}a^{n}(-b;Q)_{n}}{(ab;Q)_n(Q;Q)_n}+\sum_{n=1}^{\infty}\frac{Q^{\binom{n}{2}}a^{n}(-b;Q)_{n}}{(ab;Q)_n(Q;Q)_{n-1}}\\
=&1+\sum_{n=1}^{\infty}\frac{Q^{\binom{n}{2}}a^{n}(-b;Q)_{n}}{(ab;Q)_n(Q;Q)_{n-1}}\left(\frac{Q^{n}}{1-Q^{n}}+1\right)\\
=&1+\sum_{n=1}^{\infty}\frac{Q^{\binom{n}{2}}a^{n}(-b;Q)_{n}}{(ab;Q)_n(Q;Q)_{n}}\\
=&\sum_{n=0}^{\infty}\frac{Q^{\binom{n}{2}}a^{n}(-b;Q)_{n}}{(ab;Q)_n(Q;Q)_{n}}.
\end{align*}
So, we established the first part of this theorem, namely, the series generating function. Let $a\to \infty$, $b\to -b$, $c\to ab$ and $q\to Q$ in \eqref{eq:qGauss}, the product side follows immediately. we finish the proof.
\end{proof}

\subsection{Strict partitions with odd-indexed parts being even}

The basis for $\mathcal{G}_2$, denote by $B_{\mathcal{G}_2}$, consists of partitions in $\mathcal{G}_2$ such that $1\leq\lambda_i-\lambda_{i+1}\leq2$. Let $B_{\mathcal{G}_2}(n,h):=B_{\mathcal{G}_2}(a,b,c,d;n,h)$ be the four-parameter generating function of partitions in $B_{\mathcal{G}_2}$ with length $n$ and largest part $h$.
\begin{theorem}
The initial values of $B_{\mathcal{G}_2}(n,h)$ are
$$B_{\mathcal{G}_2}(0,h)=\left\{\begin{array}{cc}
   1  & \text{if $h=0$}, \\
   0  & \text{otherwise},
\end{array}\right.$$
and
$$B_{\mathcal{G}_2}(1,h)=\left\{\begin{array}{cc}
   ab & \text{if $h=2$},\\
   0  & \text{otherwise}.
\end{array}\right.$$
\end{theorem}

\begin{theorem}\label{thm:G2Rec}
The following recurrence relations holds for $B_{\mathcal{G}_2}(n,h)$.
\begin{equation}\label{recg2}
B_{\mathcal{G}_2}(n,2h)=abcQ^{h-1}B_{\mathcal{G}_1}(n-2,2h-2)+abQ^{h-1}B_{\mathcal{G}_1}(n-2,2h-4)   
\end{equation}
\end{theorem}
\begin{proof}
For partitions in $\mathcal{G}_2$, the largest part can only be even, so we shall only consider $B_{\mathcal{G}_2}(n,2h)$. If the largest part is $2h$, then the second largest part can be either $2h-1$ or $2h-2$. So, the weight of the first two parts is either $abcQ^{h-1}$ or $abQ^{h-1}$. By deleting these two parts from the partition, in the first case the resulted partition has length $2n-2$ and largest part $2h-2$, while in the second case it has length $2n-2$ and largest part $2h-4$. This explained the recurrence.
\end{proof}
Next, we give the closed form for $B_{\mathcal{G}_2}(n,h)$.
\begin{theorem}
For any positive integer $n$ and $h$,
\begin{equation}
B_{\mathcal{G}_2}(2n-1,2h)=Q^{\binom{n+1}{2}+\binom{h-n+1}{2}}c^{n-h-1}d^{-n}{n-1\brack h-n}_{Q}
\end{equation}
\begin{equation}
B_{\mathcal{G}_2}(2n,2h)=Q^{\binom{n+1}{2}+\binom{h-n+1}{2}}c^{n-h}d^{-n}{n\brack h-n}_{Q}.
\end{equation}
\end{theorem}

\begin{proof}
It is straightforward to check the initial values. For the recurrence, define
$$R_{\mathcal{G}_2}(2n-1,2h):=Q^{\binom{n+1}{2}+\binom{h-n+1}{2}}c^{n-h-1}d^{-n}{n-1\brack h-n}_{Q},$$
then
\begin{align*}
R_{\mathcal{G}_2}(2n-1,2h)=&Q^{\binom{n+1}{2}+\binom{h-n+1}{2}}c^{n-h-1}d^{-n}{n-1\brack h-n}_{Q}\\
=&Q^{\binom{n+1}{2}+\binom{h-n+1}{2}}c^{n-h-1}d^{-n}\left(Q^{h-n}{n-2\brack h-n}_{Q}+{n-2\brack h-n-1}_{Q}\right)\\
=&Q^{\binom{n}{2}+\binom{h-n+1}{2}+h}c^{n-h-1}d^{-n}{n-2\brack h-n}_{Q}\\
&+Q^{\binom{n}{2}+\binom{h-n}{2}+h}c^{n-h-1}d^{-n}{n-2\brack h-n-1}_{Q}\\
=&abcQ^{h-1}R_{\mathcal{G}_2}(2n-2,2h-2)+abQ^{h-1}R_{\mathcal{G}_2}(2n-2,2h-4).
\end{align*}
Similarly, define
$$R_{\mathcal{G}_2}(2n,2h):=Q^{\binom{n+1}{2}+\binom{h-n+1}{2}}c^{n-h}d^{-n}{n\brack h-n}_{Q},$$
then
\begin{align*}
R_{\mathcal{G}_2}(2n,2h)=&Q^{\binom{n+1}{2}+\binom{h-n+1}{2}}c^{n-h}d^{-n}{n\brack h-n}_{Q}\\
=&Q^{\binom{n+1}{2}+\binom{h-n+1}{2}}c^{n-h}d^{-n}\left(Q^{h-n}{n-1\brack h-n}_{Q}+{n-1\brack h-n-1}_{Q}\right)\\
=&Q^{\binom{n}{2}+\binom{h-n+1}{2}+h}c^{n-h}d^{-n}{n-1\brack h-n}_{Q}\\
&+Q^{\binom{n}{2}+\binom{h-n}{2}+h}c^{n-h}d^{-n}{n-1\brack h-n-1}_{Q}\\
=&abcQ^{h-1}R_{\mathcal{G}_2}(2n-2,2h-2)+abQ^{h-1}R_{\mathcal{G}_2}(2n-2,2h-4).
\end{align*}
We finish the proof.
\end{proof}
Now, we are ready to give the four-parameter identity for $\mathcal{G}_2$.
\begin{theorem}\label{thm:G2Four}
The four-parameter generating function for $\mathcal{G}_2$ is given by
$$\sum_{\lambda\in\mathcal{G}_{2}}\omega(\lambda)=\sum_{n=0}^{\infty}\frac{d^{-n}Q^{\binom{n+1}{2}}(-c^{-1};Q)_{n}}{(ab;Q)_{n}(Q;Q)_n}=\frac{(-abc;Q)_{\infty}}{(ab;Q)_{\infty}}.$$
\end{theorem}

\begin{proof}
By Theorem~\ref{Thm:SIP4},
\begin{align*}
\sum_{\lambda\in\mathcal{G}_{2}}\omega(\lambda)=&1+
\sum_{n=1}^{\infty}\sum_{h=0}^{\infty}\frac{B_{\mathcal{G}_2}(2n,2h)}{(ab;Q)_n(Q;Q)_n}+\sum_{n=1}^{\infty}\sum_{h=0}^{\infty}\frac{B_{\mathcal{G}_2}(2n-1,2h)}{(ab;Q)_n(Q;Q)_{n-1}}\\
=&1+\sum_{n=1}^{\infty}\frac{1}{(ab;Q)_n(Q;Q)_n}\sum_{h=0}^{\infty}Q^{\binom{n+1}{2}+\binom{h-n+1}{2}}c^{n-h}d^{-n}{n\brack h-n}_{Q}\\
&+\sum_{n=1}^{\infty}\frac{1}{(ab;Q)_n(Q;Q)_{n-1}}\sum_{h=0}^{\infty}Q^{\binom{n+1}{2}+\binom{h-n+1}{2}}c^{n-h-1}d^{-n}{n-1\brack h-n}_{Q}\\
=&1+\sum_{n=1}^{\infty}\frac{d^{-n}Q^{\binom{n+1}{2}}}{(ab;Q)_n(Q;Q)_n}\sum_{h=0}^{\infty}Q^{\binom{h+1}{2}}c^{-h}{n\brack h}_{Q}\\
&+\sum_{n=1}^{\infty}\frac{d^{-n}Q^{\binom{n+1}{2}}}{(ab;Q)_n(Q;Q)_{n-1}}\sum_{h=0}^{\infty}Q^{\binom{h+1}{2}}c^{-h-1}{n-1\brack h}_{Q}\\
&\text{(By~\eqref{eq:qBinomialFinite} with $z\to c^{-1}Q$ and $q\to Q$)}\\
=&1+\sum_{n=1}^{\infty}\frac{d^{-n}Q^{\binom{n+1}{2}}(-Q/c;Q)_{n}}{(ab;Q)_n(Q;Q)_n}+\sum_{n=1}^{\infty}\frac{c^{-1}d^{-n}Q^{\binom{n+1}{2}}(-Q/c;Q)_{n-1}}{(ab;Q)_n(Q;Q)_{n-1}}\\
=&1+\sum_{n=1}^{\infty}\frac{d^{-n}Q^{\binom{n+1}{2}}(-Q/c;Q)_{n-1}}{(ab;Q)_n(Q;Q)_{n-1}}\left(\frac{1+Q^n/c}{1-Q^n}+c^{-1}\right)\\
=&1+\sum_{n=1}^{\infty}\frac{d^{-n}Q^{\binom{n+1}{2}}(-c^{-1};Q)_{n}}{(ab;Q)_n(Q;Q)_n}\\
=&\sum_{n=0}^{\infty}\frac{d^{-n}Q^{\binom{n+1}{2}}(-c^{-1};Q)_{n}}{(ab;Q)_{n}(Q;Q)_n}.
\end{align*}
So we have the series form of the generating function as desired. Let $a\to \infty$, $b\to -c^{-1}$, $c\to ab$ and $q\to Q$ in \eqref{eq:qGauss}, the product side follows immediately. So, we finish the proof.
\end{proof}

\section{Ordinary partitions with position parity}\label{sec:Ordinary}

In this section we treat ordinary partitions with positional parity. Recall that $\mathcal{P}_1$ and $\mathcal{P}_2$ be the sets of ordinary partitions with even-indexed parts and odd-indexed parts being even, respectively. And they both are SIP classes with modulus $2$. We shall follow the same steps as in Section~\ref{sec:Strict}, giving initial values and recurrence, solving the recurrence, and then get the identity.

\subsection{Ordinary partitions with even-indexed parts being even}
Let $\mathcal{B}_{\mathcal{P}_1}$ be the basis of $\mathcal{P}_1$, then $\mathcal{B}_{\mathcal{P}_1}$ consists of partitions in $\mathcal{P}_1$ such that $0\leq\lambda_i-\lambda_{i+1}\leq1$. Let $B_{\mathcal{P}_1}(n,h):=B_{\mathcal{P}_1}(a,b,c,d;n,h)$ be the four-parameter generating function for partitions in $\mathcal{B}_{\mathcal{P}_1}$ with length $n$ and largest part $h$.
\begin{theorem}
The initial values of $B_{\mathcal{P}_1}(n,h)$ are
$$B_{\mathcal{P}_2}(0,h)=\left\{\begin{array}{cc}
   1  & \text{if $h=0$}, \\
   0  & \text{otherwise},
\end{array}\right.$$
and
$$B_{\mathcal{P}_2}(1,h)=\left\{\begin{array}{cc}
   a & \text{if $h=1$},\\
   ab & \text{if $h=2$},\\
   0  & \text{otherwise}.
\end{array}\right.$$

\end{theorem}
\begin{theorem}
For any positive integer $n$ and $h$,
\begin{equation}\label{recp11}
B_{\mathcal{P}_1}(n,2h-1)=aB_{\mathcal{P}_1}(n,2h-2),    
\end{equation}
\begin{equation}\label{recp12}
B_{\mathcal{P}_1}(n,2h)=Q^{h}B_{\mathcal{P}_1}(n-2,2h)+Q^{h}B_{\mathcal{P}_1}(n-2,2h-1).  
\end{equation}
Putting them together, we have
\begin{equation}\label{recp13}
B_{\mathcal{P}_1}(n,2h)=Q^{h}B_{\mathcal{P}_1}(n-2,2h)+aQ^{h}B_{\mathcal{P}_1}(n-2,2h-2).
\end{equation}
\end{theorem}
\begin{proof}
The proof follows a similar argument as Theorem~\ref{thm:RecG1}, so we skip it here.
\end{proof}

\begin{theorem}
For any positive integer $n$ and $h$,
\begin{equation}\label{bp1ee}
B_{\mathcal{P}_1}(2n,2h)=a^{h-1}Q^{\binom{h}{2}+n}{n-1\brack h-1}_{Q}, 
\end{equation}
\begin{equation}\label{bp1oe}
B_{\mathcal{P}_1}(2n+1,2h)=(1+b)a^{h}Q^{\binom{h}{2}+n}{n-1\brack h-1}_{Q}, 
\end{equation}
\begin{equation}\label{bp1eo}
B_{\mathcal{P}_1}(2n,2h-1)=a^{h-1}Q^{\binom{h-1}{2}+n}{n-1\brack h-2}_{Q},    
\end{equation}
\begin{equation}\label{bp1oo}
B_{\mathcal{P}_1}(2n+1,2h-1)=(1+b)a^{h}Q^{\binom{h-1}{2}+n}{n-1\brack h-2}_{Q}.    
\end{equation}
\end{theorem}

\begin{proof}
Let $R_{\mathcal{P}_1}(2n,2h)$ and $R_{\mathcal{P}_1}(2n+1,2h)$ be the right hand sides of (\ref{bp1ee}) and (\ref{bp1oe}) respectively. Then
\begin{align*}
R_{\mathcal{P}_1}(2n,2h)=&a^{h-1}Q^{\binom{h}{2}+n}{n-1\brack h-1}_{Q}\\
=&a^{h-1}Q^{\binom{h}{2}+n}\left(Q^{h-1}{n-2\brack h-1}_{Q}+{n-2\brack h-2}_{Q}\right)\quad\text{(by (\ref{QbinRec1}))}\\
=&a^{h-1}Q^{h+\binom{h}{2}+n-1}{n-2\brack h-1}_{Q}+a^{h-1}Q^{h+\binom{h-1}{2}+n-1}{n-2\brack h-2}_{Q}\\
=&Q^{h}R_{\mathcal{P}_1}(2n-2,2h)+aQ^{h}R_{\mathcal{P}_1}(2n-2,2h-2)
\end{align*}
and
\begin{align*}
R_{\mathcal{P}_1}(2n+1,2h)=&(1+b)a^{h}Q^{\binom{h}{2}+n}{n-1\brack h-1}_{Q}\\
=&(1+b)a^{h}Q^{\binom{h}{2}+n}\left(Q^{h-1}{n-2\brack h-1}_{Q}+{n-2\brack h-2}_{Q}\right)\quad\text{(by (\ref{QbinRec1}))}\\
=&(1+b)a^{h}Q^{h+\binom{h}{2}+n-1}{n-2\brack h-1}_{Q}+(1+b)a^{h}Q^{h+\binom{h-1}{2}+n-1}{n-2\brack h-2}_{Q}\\
=&Q^{h}R_{\mathcal{P}_1}(2n-1,2h)+aQ^{h}R_{\mathcal{P}_1}(2n-1,2h-2).
\end{align*}
So they satisfy (\ref{recp13}) as desired. Note that (\ref{bp1eo}) and (\ref{bp1oo}) follow from (\ref{recp11}). The initial values are easy to check, thus we finish the proof.
\end{proof}

\begin{theorem}\label{thm:P1Four}
The four-parameter generating function for $\mathcal{P}_1$ is given by
$$\sum_{\lambda\in\mathcal{P}_{1}}\omega(\lambda)=\sum_{n=0}^{\infty}\frac{Q^{n}(-aQ^{-1};Q)_{n}}{(ab;Q)_{n}(Q;Q)_{n}}+\sum_{n=0}^{\infty}\frac{abQ^{n}(-a;Q)_{n}}{(ab;Q)_{n+1}(Q;Q)_{n}}=\frac{(-a;Q)_{\infty}}{(ab;Q)_{\infty}(Q;Q)_{\infty}}.$$
\end{theorem}

\begin{proof}
Note that
\begin{align*}
\sum_{\lambda\in\mathcal{P}_1}\omega(\lambda)=&1+\sum_{n=1}^{\infty}\sum_{h=0}^{\infty}\frac{B_{\mathcal{P}_{1}}(2n,2h)+B_{\mathcal{P}_{1}}(2n,2h+1)}{(ab;Q)_{n}(Q;Q)_{n}}\\
&+\frac{a(1+b)}{1-ab}+\sum_{n=1}^{\infty}\sum_{h=0}^{\infty}\frac{B_{\mathcal{P}_{1}}(2n+1,2h)+B_{\mathcal{P}_{1}}(2n+1,2h+1)}{(ab;Q)_{n+1}(Q;Q)_{n}}.
\end{align*}
Now, since
\begin{align*}
&\sum_{n=1}^{\infty}\sum_{h=0}^{\infty}\frac{B_{\mathcal{P}_{1}}(2n,2h)+B_{\mathcal{P}_{1}}(2n,2h+1)}{(ab;Q)_{n}(Q;Q)_{n}}\\
=&\sum_{n=1}^{\infty}\frac{1}{(ab;Q)_{n}(Q;Q)_{n}}\sum_{h=0}^{\infty}\left(a^{h-1}Q^{\binom{h}{2}+n}{n-1\brack h-1}_{Q}+a^{h}Q^{\binom{h}{2}+n}{n-1\brack h-1}_{Q}\right)\\
=&\sum_{n=1}^{\infty}\frac{(1+a)Q^{n}}{(ab;Q)_{n}(Q;Q)_{n}}\sum_{h=0}^{\infty}a^{h-1}Q^{\binom{h}{2}}{n-1\brack h-1}_{Q}\\
&\text{(By~\eqref{eq:qBinomialFinite} with $z\to aQ$ and $q\to Q$)}\\
=&\sum_{n=1}^{\infty}\frac{Q^{n}(-a;Q)_{n}}{(ab;Q)_{n}(Q;Q)_{n}}
\end{align*}
and
\begin{align*}
&\sum_{n=1}^{\infty}\sum_{h=0}^{\infty}\frac{B_{\mathcal{P}_{1}}(2n+1,2h)+B_{\mathcal{P}_{1}}(2n+1,2h+1)}{(ab;Q)_{n+1}(Q;Q)_{n}}\\
=&\sum_{n=1}^{\infty}\frac{1}{(ab;Q)_{n+1}(Q;Q)_{n}}\sum_{h=0}^{\infty}\left((1+b)a^{h}Q^{\binom{h}{2}+n}{n-1\brack h-1}_{Q}+(1+b)a^{h+1}Q^{\binom{h}{2}+n}{n-1\brack h-1}_{Q}\right)\\
=&\sum_{n=1}^{\infty}\frac{(1+a)(1+b)Q^{n}}{(ab;Q)_{n+1}(Q;Q)_{n}}\sum_{h=0}^{\infty}a^{h}Q^{\binom{h}{2}}{n-1\brack h-1}_{Q}\\
&\text{(By~\eqref{eq:qBinomialFinite} with $z\to aQ$ and $q\to Q$)}\\
=&\sum_{n=1}^{\infty}\frac{aQ^{n}(1+b)(-a;Q)_{n}}{(ab;Q)_{n+1}(Q;Q)_{n}},
\end{align*}
we have
\begin{align*}
\sum_{\lambda\in\mathcal{P}_1}\omega(\lambda)=&1+\sum_{n=1}^{\infty}\frac{Q^{n}(-a;Q)_{n}}{(ab;Q)_{n}(Q;Q)_{n}}+\frac{a(1+b)}{1-ab}+\sum_{n=1}^{\infty}\frac{aQ^{n}(1+b)(-a;Q)_{n}}{(ab;Q)_{n+1}(Q;Q)_{n}}\\
=&1+\sum_{n=0}^{\infty}\frac{Q^{n+1}(-a;Q)_{n+1}}{(ab;Q)_{n+1}(Q;Q)_{n+1}}+\sum_{n=0}^{\infty}\frac{aQ^{n}(1+b)(-a;Q)_{n}}{(ab;Q)_{n+1}(Q;Q)_{n}}\\
=&1+\sum_{n=0}^{\infty}\frac{Q^{n+1}(-a;Q)_{n+1}}{(ab;Q)_{n+1}(Q;Q)_{n+1}}+\sum_{n=0}^{\infty}\frac{aQ^{n}(-a;Q)_{n}}{(ab;Q)_{n+1}(Q;Q)_{n}}+\sum_{n=0}^{\infty}\frac{abQ^{n}(-a;Q)_{n}}{(ab;Q)_{n+1}(Q;Q)_{n}}\\
=&1+\sum_{n=0}^{\infty}\frac{Q^{n}(-a;Q)_{n}}{(ab;Q)_{n+1}(Q;Q)_{n}}\left(\frac{Q(1+aQ^{n})}{1-Q^{n+1}}+a\right)+\sum_{n=0}^{\infty}\frac{abQ^{n}(-a;Q)_{n}}{(ab;Q)_{n+1}(Q;Q)_{n}}\\
=&1+\sum_{n=0}^{\infty}\frac{Q^{n}(-a;Q)_{n}}{(ab;Q)_{n+1}(Q;Q)_{n}}\left(\frac{a+Q}{1-Q^{n+1}}\right)+\sum_{n=0}^{\infty}\frac{abQ^{n}(-a;Q)_{n}}{(ab;Q)_{n+1}(Q;Q)_{n}}\\
=&1+\sum_{n=0}^{\infty}\frac{Q^{n+1}(-aQ^{-1};Q)_{n+1}}{(ab;Q)_{n+1}(Q;Q)_{n+1}}+\sum_{n=0}^{\infty}\frac{abQ^{n}(-a;Q)_{n}}{(ab;Q)_{n+1}(Q;Q)_{n}}\\
=&\sum_{n=0}^{\infty}\frac{Q^{n}(-aQ^{-1};Q)_{n}}{(ab;Q)_{n}(Q;Q)_{n}}+\sum_{n=0}^{\infty}\frac{abQ^{n}(-a;Q)_{n}}{(ab;Q)_{n+1}(Q;Q)_{n}}.
\end{align*}
So we have the series expression for the generating function. Next we establish the product side. For any integer $N\geq0$, let
$$F_{1}(N):=\sum_{n=0}^{N}\frac{Q^{n}(-aQ^{-1};Q)_{n}}{(ab;Q)_{n}(Q;Q)_{n}}+\sum_{n=0}^{N}\frac{abQ^{n}(-a;Q)_{n}}{(ab;Q)_{n+1}(Q;Q)_{n}}$$
and
$$T_1(N):=\frac{(-a;Q)_{N}}{(ab;Q)_{N+1}(Q;Q)_{N}}.$$
It's straight forward to see that
$$F_1(0)=1+\frac{ab}{1-ab}=\frac{1}{1-ab}=T_1(0).$$
Note that
\begin{align*}
&T_1(N+1)-T_1(N)\\
=&\frac{(-a;Q)_{N+1}}{(ab;Q)_{N+2}(Q;Q)_{N+1}}-\frac{(-a;Q)_{N}}{(ab;Q)_{N+1}(Q;Q)_{N}}\\
=&\frac{(-a;Q)_{N}}{(ab;Q)_{N+1}(Q;Q)_{N}}\left(\frac{1+aQ^{N}}{(1-abQ^{N+1})(1-Q^{N+1})}-1\right)\\
=&\frac{(-a;Q)_{N}}{(ab;Q)_{N+1}(Q;Q)_{N}}\times\frac{aQ^{N}+Q^{N+1}+abQ^{N+1}-abQ^{2N+2}}{(1-abQ^{N+1})(1-Q^{N+1})}\\
=&\frac{(-a;Q)_{N}}{(ab;Q)_{N+1}(Q;Q)_{N}}\times\frac{aQ^{N}+Q^{N+1}+abQ^{N+1}-abQ^{2N+2}+a^2bQ^{2N+1}-a^2bQ^{2N+1}}{(1-abQ^{N+1})(1-Q^{N+1})}\\
=&\frac{(-a;Q)_{N}}{(ab;Q)_{N+1}(Q;Q)_{N}}\times\frac{Q^{N+1}(1+aQ^{-1})(1-abQ^{N+1})+abQ^{N+1}(1+aQ^{N})}{(1-abQ^{N+1})(1-Q^{N+1})}\\
=&\frac{Q^{N+1}(-aQ^{-1};Q)_{N+1}}{(ab;Q)_{N+1}(Q;Q)_{N+1}}+\frac{abQ^{N+1}(-a;Q)_{N+1}}{(ab;Q)_{N+2}(Q;Q)_{N+1}}\\
=&F_1(N+1)-F_1(N).
\end{align*}
So we have shown that $F_1(N)=T_1(N)$ for all $N\geq0$ by induction. Sending $N\to\infty$, we finish the proof.
\end{proof}

\subsection{Ordinary partitions with odd-indexed parts being even}

It is clear that the basis for $\mathcal{P}_{2}$, denoted by $\mathcal{B}_{\mathcal{P}_{2}}$, consists of all the partitions in $\mathcal{P}_{2}$ such that $0\leq\lambda_i-\lambda_{i+1}\leq1$. Let $B_{\mathcal{P}_2}(n,h)=B_{\mathcal{P}_2}(a,b,c,d;n,h)$ be the four-parameter generating function for partitions in $\mathcal{B}_{\mathcal{P}_2}$ with length $n$ and largest part $h$.
\begin{theorem}
The initial values of $B_{\mathcal{P}_2}(n,h)$ are
$$B_{\mathcal{P}_2}(0,h)=\left\{\begin{array}{cc}
   1  & \text{if $h=0$}, \\
   0  & \text{otherwise},
\end{array}\right.$$
and
$$B_{\mathcal{P}_2}(1,h)=\left\{\begin{array}{cc}
   ab & \text{if $h=2$},\\
   0  & \text{otherwise}.
\end{array}\right.$$
\end{theorem}
\begin{theorem}
For any positive integer $n$ and $h$, we have
\begin{equation}\label{recp2}
B_{\mathcal{P}_2}(n,2h)=Q^{h}B_{\mathcal{P}_2}(n-2,2h)+d^{-1}Q^{h}B_{\mathcal{P}_2}(n-2,2h-2).  
\end{equation}
\end{theorem}
\begin{proof}
The proof follows a similar argument as Theorem~\ref{thm:G2Rec}, so we skip it here.
\end{proof}

\begin{theorem}
For any positive integer $n$ and $h$,
\begin{equation}\label{bp2ee}
B_{\mathcal{P}_2}(2n,2h)=Q^{\binom{h}{2}+n}(1+d^{-1})d^{1-h}{n-1\brack h-1}_{Q},
\end{equation}
\begin{equation}\label{bp2oe}
B_{\mathcal{P}_2}(2n-1,2h)=abQ^{\binom{h}{2}+n-1}d^{1-h}{n-1\brack h-1}_{Q}.    
\end{equation}
\end{theorem}

\begin{proof}
The initial values are trivial to check. It suffices to show that they satisfy the recurrence relation in (\ref{recp2}). Denote by $R_{\mathcal{P}_2}(2n,2h)$ and $R_{\mathcal{P}_2}(2n-1,2h)$ the right hand side of (\ref{bp2ee}) and (\ref{bp2oe}) respectively, then
\begin{align*}
R_{\mathcal{P}_2}(2n,2h)=&Q^{\binom{h}{2}+n}(1+d^{-1})d^{1-h}{n-1\brack h-1}_{Q}\\
=&Q^{\binom{h}{2}+n}(1+d^{-1})d^{1-h}\left(Q^{h-1}{n-2\brack h-1}_{Q}+{n-2\brack h-2}_{Q}\right)\quad\text{(by (\ref{QbinRec1}))}\\
=&Q^{h+\binom{h}{2}+n-1}(1+d^{-1})d^{1-h}{n-2\brack h-1}_{Q}\\
&+d^{-1}Q^{h+\binom{h-1}{2}+n-1}(1+d^{-1})d^{2-h}{n-2\brack h-1}_{Q}\\
=&Q^{h}R_{\mathcal{P}_2}(2n-2,2h)+d^{-1}Q^{h}R_{\mathcal{P}_2}(2n-2,2h-2)
\end{align*}
and
\begin{align*}
R_{\mathcal{P}_2}(2n-1,2h)=&abQ^{\binom{h}{2}+n-1}d^{1-h}{n-1\brack h-1}_{Q}\\
=&abQ^{\binom{h}{2}+n-1}d^{1-h}\left(Q^{h-1}{n-2\brack h-1}_{Q}+{n-2\brack h-2}_{Q}\right)\quad\text{(by (\ref{QbinRec1}))}\\
=&abQ^{h+\binom{h}{2}+n-2}d^{1-h}{n-2\brack h-1}_{Q}\\
&+abd^{-1}Q^{h+\binom{h-1}{2}+n-2}d^{2-h}{n-2\brack h-1}_{Q}\\
=&Q^{h}R_{\mathcal{P}_2}(2n-3,2h)+d^{-1}Q^{h}R_{\mathcal{P}_2}(2n-3,2h-2).
\end{align*}
One can easily check the initial value, and so we finish the proof.
\end{proof}

\begin{theorem}\label{thm:P2Four}
The four-parameter generating function for $\mathcal{P}_2$ is given by
$$\sum_{\lambda\in\mathcal{P}_{2}}\omega(\lambda)=\sum_{n=0}^{\infty}\frac{Q^{n}(-1/d;Q)_{n}}{(ab;Q)_{n}(Q;Q)_{n}}+\sum_{n=0}^{\infty}\frac{abQ^{n}(-Q/d;Q)_{n}}{(ab;Q)_{n+1}(Q;Q)_{n}}=\frac{(-Q/d;Q)_{\infty}}{(ab;Q)_{\infty}(Q;Q)_{\infty}}.$$
\end{theorem}

\begin{proof}
Note that
\begin{align*}
\sum_{\lambda\in\mathcal{P}_2}\omega(\lambda)=&1+\sum_{n=1}^{\infty}\sum_{h=0}^{\infty}\frac{B_{\mathcal{P}_{2}}(2n,2h)}{(ab;Q)_{n}(Q;Q)_{n}}+\sum_{n=0}^{\infty}\sum_{h=0}^{\infty}\frac{B_{\mathcal{P}_{1}}(2n+1,2h)}{(ab;Q)_{n+1}(Q;Q)_{n}}\\
=&1+\sum_{n=1}^{\infty}\frac{(1+d^{-1})Q^{n}}{(ab;Q)_{n}(Q;Q)_{n}}\sum_{h=0}^{\infty}Q^{\binom{h}{2}}d^{1-h}{n-1\brack h-1}_{Q}\\
&+\sum_{n=0}^{\infty}\frac{abQ^{n}}{(ab;Q)_{n+1}(Q;Q)_{n}}\sum_{h=0}^{\infty}Q^{\binom{h}{2}}d^{1-h}{n\brack h-1}_{Q}\\
&\text{(By~\eqref{eq:qBinomialFinite} with $z\to d^{-1}Q$ and $q\to Q$)}\\
=&1+\sum_{n=1}^{\infty}\frac{(1+d^{-1})Q^{n}(-Q/d;Q)_{n-1}}{(ab;Q)_{n}(Q;Q)_{n}}+\sum_{n=0}^{\infty}\frac{abQ^{n}(-Q/d;Q)_{n}}{(ab;Q)_{n+1}(Q;Q)_{n}}\\
=&\sum_{n=0}^{\infty}\frac{Q^{n}(-1/d;Q)_{n}}{(ab;Q)_{n}(Q;Q)_{n}}+\sum_{n=0}^{\infty}\frac{abQ^{n}(-Q/d;Q)_{n}}{(ab;Q)_{n+1}(Q;Q)_{n}}.
\end{align*}
So we proved the series expression for the generating function. For the product side, let
$$F_{2}(N):=\sum_{n=0}^{N}\frac{Q^{n}(-1/d;Q)_{n}}{(ab;Q)_{n}(Q;Q)_{n}}+\sum_{n=0}^{N}\frac{abQ^{n}(-Q/d;Q)_{n}}{(ab;Q)_{n+1}(Q;Q)_{n}}$$
and
$$T_2(N):=\frac{(-Q/d;Q)_{N}}{(ab;Q)_{N+1}(Q;Q)_{N}}.$$
It's straight forward to see that
$$F_2(0)=1+\frac{ab}{1-ab}=\frac{1}{1-ab}=T_2(0).$$
Note that
\begin{align*}
T_2(N+1)-T_2(N)=&\frac{(-Q/d;Q)_{N+1}}{(ab;Q)_{N+2}(Q;Q)_{N+1}}-\frac{(-Q/d;Q)_{N}}{(ab;Q)_{N+1}(Q;Q)_{N}}\\
=&\frac{(-Q/d;Q)_{N}}{(ab;Q)_{N+1}(Q;Q)_{N}}\left(\frac{1+Q^{N+1}/d}{(1-abQ^{    N+1})(1-Q^{N+1})}-1\right)\\
=&\frac{(-Q/d;Q)_{N}}{(ab;Q)_{N+1}(Q;Q)_{N}}\left(\frac{Q^{N+1}(1+1/d)}{(1-Q^{N+1})}+\frac{abQ^{N+1}(1+Q^{N+1}/d)}{(1-abQ^{N+1})(1-Q^{N+1})}\right)\\
=&\frac{Q^{N+1}(-1/d;Q)_{N+1}}{(ab;Q)_{N+1}(Q;Q)_{N+1}}+\frac{abQ^{N+1}(-Q/d;Q)_{N+1}}{(ab;Q)_{N+2}(Q;Q)_{N+1}}\\
=&F_2(N+1)-F_2(N).
\end{align*}
By induction, we have $F_2(N)=T_2(N)$ for all $N\geq0$. Let $N\to\infty$, we finish the proof.    
\end{proof}

\section{Conclusion}\label{sec:Comment}

We have treated partitions with positional parity conditions as examples, and based on the four-parameter identities we found, the unification of Theorem~\ref{thm:SavageSillsBivariable} and Theorem~\ref{thm:LiBivariable} follows immediately.
\begin{theorem}
\begin{equation}
\sum_{\lambda\in\mathcal{G}_1}x^{o(\lambda)}z^{a(\lambda)}q^{|\lambda|}=\sum_{n=0}^{\infty}\frac{x^{n}z^{n}q^{2n^2-n}(-zq/x;q^4)_{n}}{(z^2q^2;q^4)_{n}(q^4;q^4)_{n}}=\frac{(-xzq;q^4)_{\infty}}{(z^2q^2;q^4)_{\infty}},
\end{equation}
\begin{equation}
\sum_{\lambda\in\mathcal{G}_2}x^{o(\lambda)}z^{a(\lambda)q^{|\lambda|}}=\sum_{n=0}^{\infty}\frac{x^{n}z^{n}q^{2n^2+n}(-zq^{-1}/x;q^4)_{n}}{(z^2q^2;q^4)_{n}(q^4;q^4)_{n}}=\frac{(-xzq^3;q^4)_{\infty}}{(z^2q^2;q^4)_{\infty}}.
\end{equation}
\end{theorem}
\begin{proof}
Let $a\to xzq$, $b\to zq/x$, $c\to xq/z$ and $d\to q/xz$ in Theorem~\ref{Thm:G1Four} and Theorem~\ref{thm:G2Four}, respectively. Then, we have the desired identities.
\end{proof}
With the same substitution, we also have the following identities derived from Theorem~\ref{thm:P1Four} and Theorem~\ref{thm:P2Four}.
\begin{theorem}
\begin{equation}
\begin{split}
\sum_{\lambda\in\mathcal{P}_{1}}x^{o(\lambda)}z^{a(\lambda)}q^{|\lambda|}=&\sum_{n=0}^{\infty}\frac{q^{4n}(-xzq^{-3};q^4)_{n}}{(z^2q^2;q^4)_{n}(q^4;q^4)_{n}}+\sum_{n=0}^{\infty}\frac{z^{2}q^{4n+2}(-xzq;q^4)_{n}}{(z^2q^2;q^4)_{n+1}(q^4;q^4)_{n}}\\
=&\frac{(-zxq;q^4)_{\infty}}{(z^2q^2;q^4)_{\infty}(q^4;q^4)_{\infty}},
\end{split}
\end{equation}
\begin{equation}
\begin{split}
\sum_{\lambda\in\mathcal{P}_{2}}x^{o(\lambda)}z^{a(\lambda)}q^{|\lambda|}=&\sum_{n=0}^{\infty}\frac{q^{4n}(-xz/q;q^4)_{n}}{(z^2q^2;q^4)_{n}(q^4;q^4)_{n}}+\sum_{n=0}^{\infty}\frac{z^2q^{4n}(-xzq^3;Q)_{n}}{(z^2q^2;q^4)_{n+1}(q^4;q^4)_{n}}\\
=&\frac{(-xzq^3;q^4)_{\infty}}{(z^2q^2;q^4)_{\infty}(q^4;q^4)_{\infty}}.
\end{split}
\end{equation}
\end{theorem}

The purpose of this work is to extend the idea of SIP classes and provide a unified method to study the four-parameter weight of partitions. This weight has been the subject of a number of studies, see \cite{BerkovichUncuSchmidt,BerkovichUncu,FuZeng} for example. In~\cite{BerkovichUncu}, the author used the four-parameter weight to study the BG-rank of partition, which was first introduced by Berkovich and Garvan in~\cite{BerkovichGarvan}. Given a partition $\lambda$, the BG-rank of $\lambda$, denoted by $\text{BG}(\lambda)$, is the number of odd-indexed odd parts minus the number of even-indexed odd parts. This statistic has rich properties and an interesting generalization, called the GBG-rank, see~\cite{BerkovichDhar,BerkovichGarvan2,BerkovichGarvanGBG,Yee} for examples. Using our four-parameter identity, we can also provide the following generating functions.
\begin{theorem}
\begin{equation}
\sum_{\lambda\in\mathcal{G}_1}z^{\text{BG}(\lambda)}q^{|\lambda|}=\sum_{n=0}^{\infty}\frac{z^{n}q^{2n^2-n}(-zq;q^4)_{n}}{(q^2;q^2)_{2n}}=\frac{(-zq;q^4)_{\infty}}{(q^2;q^4)_{\infty}},
\end{equation}
\begin{equation}
\sum_{\lambda\in\mathcal{G}_2}z^{\text{BG}(\lambda)}q^{|\lambda|}=\sum_{n=0}^{\infty}\frac{z^{-n}q^{2n^2+n}(-zq^{-1};q^4)_{n}}{(q^2;q^2)_{2n}}=\frac{(-q^3/z;q^4)_{\infty}}{(q^2;q^4)_{\infty}}.
\end{equation}
\end{theorem}
\begin{proof}
Let $a\to zq$, $b\to q/z$, $c\to q/z$ and $d\to zq$ in Theorem~\ref{Thm:G1Four} and Theorem~\ref{thm:G2Four}, we have the results immediately.
\end{proof}
\begin{theorem}
\begin{equation}
\sum_{\lambda\in\mathcal{P}_{1}}z^{\text{BG}(\lambda)}q^{|\lambda|}=\sum_{n=0}^{\infty}\frac{q^{4n}(-zq^{-3};q^{4})_{n}}{(q^{2};q^{2})_{2n}}+\sum_{n=0}^{\infty}\frac{q^{4n+2}(-zq;q^{4})_{n}}{(q^2;q^2)_{2n+1}}=\frac{(-zq;q^4)_{\infty}}{(q^2;q^2)_{\infty}},
\end{equation}
\begin{equation}
\sum_{\lambda\in\mathcal{P}_{2}}z^{\text{BG}(\lambda)}q^{|\lambda|}=\sum_{n=0}^{\infty}\frac{q^{4n}(-z^{-1}q^{-1};q^{4})_{n}}{(q^{2};q^{2})_{2n}}+\sum_{n=0}^{\infty}\frac{q^{4n+2}(-z^{-1}q^{3};Q)_{n}}{(q^{2};q^{2})_{2n+1}}=\frac{(-z^{-1}q^3;q^4)_{\infty}}{(q^{2};q^{2})_{\infty}}.    
\end{equation}
\begin{proof}
Let $a\to zq$, $b\to q/z$, $c\to q/z$ and $d\to zq$ in Theorem~\ref{thm:P1Four} and Theorem~\ref{thm:P2Four}, we finish the proof.
\end{proof}
\end{theorem}
Recently, there have been some new discovery of SIP classes, see~\cite{ChenHeTangWei,HeHuangLiZhang} for more details. It would be interesting to seek for the four-parameter extension for the SIP classes in their work. Also, we mentioned that Boulet in~\cite{Boulet} provided the product form fou-parameter generating function for the set of all partitions and the set of strict partitions. In~\cite{BerkovichUncu}, the authors gave a series form by the recurrence on the largest part. It would also be interesting to see what kind of generating we can obtain by SIP classes.





\end{document}